\documentclass[12 pt]{amsart}

\newtheorem{theorem}{Theorem}[section]
\newtheorem{lemma}[theorem]{Lemma}

\newtheorem{conjecture}[theorem]{Conjecture}

\theoremstyle{definition}

\theoremstyle{remark}

\usepackage{amscd,amssymb}

\begin{document}
\baselineskip 16pt

\title[Applications of degree estimate for subalgebras]
{Applications of degree estimate for subalgebras}

\author[Yun-Chang Li and Jie-Tai Yu]
{Yun-Chang Li and Jie-Tai Yu}
\address{Department of Mathematics, The University of Hong
Kong, Hong Kong SAR, China} \email{liyunch@hku.hk,\ liyunch@163.com}
\address{Department of Mathematics, The University of Hong
Kong, Hong Kong SAR, China} \email{yujt@hkusua.hku.hk,\
yujietai@yahoo.com}


\thanks{The research of Jie-Tai Yu was partially
supported by an RGC-CERG Grant.}




\begin{abstract}
Let $K$ be a field of positive characteristic and $K\langle x, y\rangle$ be the free algebra of rank two over $K$. Based on the degree estimate done by Y.-C. Li and J.-T. Yu, we extend the results of S.J. Gong and J.T. Yu's results: (1) An element $p(x,y)\in K\langle x,y\rangle$ is a test element if and only if $p(x,y)$ does not belong to any proper retract of $K\langle x,y\rangle$; (2) Every endomorphism preserving the automorphic orbit of a nonconstant element of $K\langle x,y\rangle$ is an automorphism; (3) If there exists some injective endomorphism  $\phi$ of $K\langle x,y\rangle$ such that $\phi(p(x,y))=x$ where $p(x,y)\in K\langle x,y\rangle$, then $p(x,y)$ is a coordinate. And we reprove that all the automorphisms of $K\langle x,y\rangle$ are tame. Moreover, we also give counterexamples for two conjectures established by Leonid Makar-Limanov, V. Drensky and J.-T. Yu in the positive characteristic case.
\end{abstract}

\maketitle

\noindent Let $K$ be a field   and $A_n=K[x_1,\cdots,x_n]$ or $A_n=K\langle x_1,\cdots,x_n\rangle$. A polynomial
$p\in A_n$ is called a test element of $A_n$, provided for all $K$-endomorphisms $\phi$ of $A_n$ fixing $p$ are  automorphisms of $A_n$.  A subalgebra $R$ is a retract of $A_n$ if there exists a $K$-endomorphism $\pi$ of $A_n$ such that $R=\pi(A_n)$ and $\pi$ fixes every element of $R$. Test elements and retracts of groups and other
algebraic structures are defined in a similar way. Test elements and retracts of algebras and groups have recently been studied in \cite{C,DY2,DY3,ES,J2,MSY,MUY2,MY1,MY2,MY3,MZ,S1,S2,SY1,SY2}.

\

\noindent It is easy to know by the definition that a test element does not belong to any proper retract of $A_n$.  The converse is proved by Turner \cite{T} for free groups, by Mikhalev and Zolotykh \cite{MZ} and by Mikhalev and J.-T. Yu \cite{MY2,MY3} for free Lie algebras and free Lie superalgebras respectively, by Mikhalev, Umirbaev and J.-T. Yu \cite{MUY1} for free nonassociative algebras, and by S.-J. Gong and J.-T. Yu \cite{GY} for $A_2$ in characteristic zero case. 

\

\noindent Now consider another related topic.  The set $S(p)=\{\phi(p)|\phi\in Aut(A_n)$, $p\in A_n\}$ is called the automorphic orbit of $p$. In a free group, the automorphic orbit of some element is defined similarly. Certainly an automorphism preserves the automorphic orbit of an element. The converse is proved by Shpilrain \cite{S3} and Ivanov \cite{I}
for free groups of rank two, by D. Lee \cite{L} for free groups of any rank, by Mikhalev and J.-T. Yu \cite{MY3} for free Lie algebras, by Mikhalev, Umirbaev and J.-T. Yu \cite{MUY1} for free nonassociative algebras, by van den Essen and Shpilrain \cite{ES} for $A_2$ when $p$ is a coordinate, by Jelonek \cite{J1} for polynomial algebras over $C$ when $p$ is a coordinate and by S.-J. Gong and J.-T. Yu \cite{GY} for $A_2$ in characteristic zero case. 
 
 \
 
 \noindent In section 1 we extend S.-J. Gong and J.-T. Yu's above results
 regarding  test elements, retracts,
 and automorphic orbits to the free associative algebra $A_2=K\langle x,y\rangle$ over a field $K$ of positive characteristic. Most of the steps are the same, and the only difference is that we give a new method instead of using the AMS theorem since it is not true for positive characteristic case.
 
\
 
\noindent  The following degree estimate \cite{LY} in arbitrary characteristic
plays a crucial role in this paper.

\

\begin{theorem}[Degree estimate]\label{th1}
Let $A_n=K\langle x_1,\cdots,x_k\rangle$ be a free associative algebra over a field $K$ of arbitrary characteristic and let $f,g\in A$ be algebraically independent. Suppose that $v(f)$ and $v(g)$ are algebraically dependent and neither $deg(f)$ divides $deg(g)$ nor $deg(g)$ divides $deg(f)$. Then for any $P\in F\langle x,y\rangle$ we have
$$deg(P(f,g))\geq D(f,g)w_{deg(f),deg(g)}(P)\,\,\,\,where\,\,\,\,D(f,g)=\frac{deg([f,g])}{deg(fg)}$$
\end{theorem}

\noindent In section 2, we consider another topic. J.-T. Yu and Makar-Limanov raised the following two conjectures in \cite{Y2} and in \cite{DY1} respectively.

\begin{conjecture}[J.-T. Yu]\label{con1} Let $K$ be a field of characteristic zero and $f$ and $g$ be algebraically independent polynomials in $K\langle x,y\rangle$ such that the homogeneous components of maximal degree of $f$ and $g$ are algebraically dependent. Let $f$ and $g$ generate their own centralizers in $K\langle X\rangle$ respectively. Suppose that $\deg(f)\nmid \deg(g),deg(g)\nmid \deg(f)$. Then
$$deg([f,g])>min\{deg(f),deg(g)\}.$$
\end{conjecture}

\begin{conjecture}[Makar-Limanov and J.-T. Yu]\label{con2}
Let $Char(K)=0$, $g\in K\langle X\rangle$ generate its own centralizer and let the homogeneous component of maximal degree of $g$ is an $n$-th power of an element of $K\langle X\rangle$. Then, for every $m>n$ which is not divisible by $n$, the formal power series $g^{m/n}\in K((X))$ has a monomial of positive degree containing a negative power of an indeterminate in $X$.
\end{conjecture}

\noindent If Conjecture 0.2 is true, it would give a nice description of the group of tame automorphisms of $K\langle x,y,z\rangle$ algorithmically, much better than the description of the group of tame automorphisms of $K[x,y,z]$(see \cite{Y2}). 

\

\noindent In \cite{DY1} Vesselin Drensky and J.-T. Yu analyze the structure of $K\langle X\rangle$ as a $K[u]$-bimodule in characteristic zero case for some fixed monomial $u\in K\langle X\rangle$ which is not a proper power of another monomial and give the counterexamples for both conjectures. 

\

\noindent However, in the positive characteristic case, Conjecture 0.3 no long makes sense since even the leading monomial of $g$ is an $n$-th power, $g$ may not be an $n$-th power in $K((X))$.  Hence, in section 2, we modify Conjecture 0.3 to make sense for arbitrary characteristic and if restricted to characteristic zero, it is equivalent to Conjecture 0.3. Then for positive characteristic, we give counterexamples for both conjectures.

\section{Test elements, retracts and automorphic orbits}

\noindent Recall that $p\in A_2$ is called of outer rank 2 if to each pair coordinates $(f,g)$ of $A_2$, both $f$ and $g$ appears in $h(f,g)=p(x,y)$. Then we have

 \begin{lemma}\label{lem1}
  Let $K$ be an arbitrary field. If $p(x,y)\in A_2$ does not belong to any proper retract, then it is of outer rank 2.
 \end{lemma}
 
 \begin{proof}
 If $p$ is of outer rank 1, let $p(x,y)=h(f)$ where $h(u)\in K[u]$ and $(f,g)$ is a pair of coordinates of $A_2$. Then $\phi: f\mapsto f; g\mapsto 0$ is a retraction fixing $p(x,y)$ with the corresponding retract $K[f]$. This contradicts, and hence $p$ is of outer rank 2.
\end{proof}

\begin{lemma}\label{lem3}
Let $K$ be an arbitrary field, $\phi$ be a proper retraction fixing $p(x,y)\in A_2-K$ and the corresponding retract is $R$. Then $\phi$ also fixes a primitive element $r(x,y)$ of $p(x,y)$(a primitive element $r(x,y)$ means that if $r(x,y)=h(r_1(x,y))$ for some polynomial $h(u)\in K[u]$ and $r_1(x,y)\in A_2$, then $\deg(h(u))=1$. A primitive element $r(x,y)$ of $p(x,y)$ means that $p(x,y)=h(r)$ for some $h(u)\in K[u]$ and $r$ is primitive). Moreover, $R=K[r]$.
\end{lemma}

 \begin{proof}

It is proved by Casta in \cite{C} (Lemma 1.3 and Theorem 3.5).

 \end{proof}

\noindent The proofs for the following five lemmas are similar to \cite{GY} 
for zero characteristic case. Here we just using the degree estimate in \cite{LY} instead of \cite{MLY2}.

\begin{lemma}\label{lem2}
Let $K$ be an arbitrary field, $p(x,y)\in A_2-K$, and $\phi$ is a noninjective endomorphism of $A_2$. If $\phi(p)=p$, then $\phi^m$ is a retraction fixing $p(x,y)$ for some positive integer $m$.
\end{lemma}

\begin{lemma}\label{lem4}
Let $f(x,y),g(x,y)\in K\langle x,y\rangle$ and $p(x,y)\in K\langle x,y\rangle$ be of outer rank 2. Then $w_{deg(f),deg(g)}(p(x,y))\geq \deg(f)+\deg(g)$. Moreover, if there exists a monomial in $p(x,y)$ whose degree is greater than 2 containing both $x$ and $y$, then $w_{\deg(f),\deg(g)}(p(f,g))> \deg(f)+\deg(g)$.
\end{lemma}

\begin{lemma}\label{lem5}
  Let $\phi=(f,g)$ be an injective endomorphism of $K\langle x,y\rangle$ and $p(x,y)\in K\langle x,y\rangle$ is of outer rank 2. Then $\deg(p(f,g))\geq \deg([f,g])$.
\end{lemma}

\begin{lemma}\label{lem6}
Let $\phi=(f,g)$ be an injective endomorphism but not an automorphism of $K\langle x,y\rangle$. Then $[\phi^k(x),\phi^k(y)]\geq k+2$ for all $k\in N$.
\end{lemma}

\begin{lemma}\label{lem7}
Let $p(x,y)$ be of outer rank 2. Then if $\phi=(f,g)$ is an injective endomorphism but not an automorphism, $\phi(p(x,y))\not=p(x,y)$.
\end{lemma}

\noindent Now we are ready to give the  main results of this section.

\begin{theorem}\label{th2}
An element $p(x,y)\in K\langle x,y\rangle$ is a test polynomial  if and only if $p(x,y)$ doesn't belong to any proper retract.
\end{theorem}

\begin{proof}
The same as that \cite{GY} for zero characteristic case.
\end{proof}

\begin{theorem}\label{th3}
If an endomorphism $\phi$ preserve the automorphic orbit of a nonconstant element $p\in K\langle x,y\rangle$, then $\phi$ is an automorphism of $K\langle x,y\rangle$.
\end{theorem}

\begin{proof}
We may assume $\phi(p(x,y))=p(x,y)$. By the definition of test elements, we may assume $p(x,y)$ is not a test element. By Theorem \ref{th2} , we may assume $p(x,y)$ belongs to some proper retract, by a result of J.-T. Yu(\cite{Y1}, Lemma 2.3), we may assume $p(x,y)$ is of outer rank 2. By Lemma \ref{lem7} , we may assume $\phi$ is noninjective. Suppose that $p=f(r)$, where $f\in K[t]-K$ and $r$ is primitive. By Lemma \ref{lem2} , $\pi=\phi^m$ is a retraction of $K\langle x,y\rangle$ onto $K[r]$ for some positive integer $m$. As $\phi$ preserves the automorphic orbit of $p$, so does $\pi=\phi^m$. Hence, we have reduced our proof to the following.
\end{proof}

\begin{lemma}\label{lem8}
Let $\pi=(h_1(r),h_2(r))$ be a proper retraction of $K\langle x,y\rangle$ onto $K[r]$ for some primitive element $r$. Then $\pi$ does not preserve the automorphic orbit of $f(r)$ where $s(x,y)=f(r)$ is of outer rank 2.
\end{lemma}

\begin{proof}
Suppose on the contrary, namely $\pi$ preserves the automorphic orbit of $f(r)$. Then to each automorphism $\alpha$ there exists some automorphism $\beta$ such that $\pi\alpha(f(r))=\beta(f(r))$, or $f(\pi\alpha(r))=f(\beta(r))$ equivalently. Since $f(t)\not\in K$, $\pi\alpha(r)$ and $\beta(r)$ are algebraically dependent and hence they are generated by one element. Notice here, since $r$ is primitive, so is $\beta(r)$. If not, then there exists some $h(t)\in K[t]$ with $\deg_t(h(t))\geq 2$ and $r_1\in K\langle x,y\rangle$ such that $\beta(r)=h(r_1)$. Then $r=h(\beta^{-1}(r_1))$. Since $\deg_t(h(t))\geq 2$, it contradicts to $r$ being primitive. Hence $\beta(r)$ is also primitive. Then, since $\pi\alpha(r)$ and $\beta(r)$ are algebraically dependent, $\pi\alpha(r)=g(\beta(r))$ for some $g(t)\in K[t]$, and hence $g(\beta(r))\in K[r]$ since $\pi\alpha(r)=\pi(\alpha(r))\in K[r]$. 

\noindent By Lemma \ref{lem3} , $\beta(r)\in K[r]$ as well, and since $r$ generates $\beta(r)$ as well as $\beta(r)$ being primitive again, $\beta(r)=ar+b$ where $a,b\in K$, $a\not=0$. 

\noindent Hence, to any automorphism $\alpha$ of $K\langle x,y\rangle$, $\pi\alpha(f(r))=f(ar+b)$ for some $a,b\in K$, $a\not=0$, and since $\deg_r(f(\pi\alpha(r)))=\deg_t(f(t))\cdot \deg_r(\pi\alpha(r))$, $\deg_r(f(ar+b))=\deg_t(f)\cdot \deg_r(ar+b)$, we have $\deg_r(\pi\alpha(r))$ $=\deg_r(ar+b)=1$. Now we prove this is impossible.

\noindent If $r$ is of outer rank 1, then each element of $K[r]$ is of outer rank 1. This contradicts to $s(x,y)=f(r)$ being of outer rank 2, and hence $r$ is of outer rank 2. Then both of $x$ and $y$ appear in $r(x,y)$. Since $K[r]$ is proper, not both of $h_1(r)$ and $h_2(r)$ are constants, so we can assume $\deg(h_2(r(x,y)))\geq 1$. Let $r(x,y)=\sum_{i=0}^ka_i(y)x^i$ where $a_k(y)\not=0$, $k\geq 1$ and $\alpha_M: x\mapsto x+y^M; y\mapsto y$ be an automorphism of $K\langle x,y\rangle$ where $M$ can be any non-negative integer. Then 
$$\alpha_M(r)=\sum_{i=0}^ka_i(y)(x+y^M)^i.$$
We rewrite the relation as 
$$\alpha_M(r)=\sum_{i=0}^kb_i(x,y)y^{Mi}$$
where $b_i(x,y)$ is not concerned with $M$ for each $i$. Notice here, $b_k(x,y)=a_k(y)$, and we can also rewrite the relation as 
$$\alpha_M(r)=a_k(y)y^{kM}+\sum_{i=0}^{k-1}b_i(x,y)y^{Mi},$$
Then 
$$\begin{array}{rcl}\pi\alpha_M(r)&=&\pi(a_k(y)y^{kM}+\sum_{i=0}^{k-1}b_i(x,y)y^{Mi})\\&=&a_k(h_2(r))h_2^{kM}(r)+\sum_{i=0}^{k-1}b_i(h_1(r),h_2(r))h_2^{Mi}(r).\end{array}$$
Since $a_k(y)\not=0$ as well as $h_2(r)\not\in K$, $a_k(h_2(r))\not=0$. Then since $b_i(h_1(r),h_2(r))$ is not concerned with $M$, we can choose $M$ great enough such that 
$$\begin{array}{rcl}\deg_r(\pi\alpha_M(r))&=&\deg_r(a_k(h_2(r))h_2^{kM}(r))\\&=&\deg(a_k(y))\cdot\deg_r(h_2(r))+kM\cdot\deg_r(h_2(r))\\&\geq& M\\&>&1,\end{array}$$
 
\noindent namely there exists an automorphism $\alpha_M$ such that $\pi\alpha_M(f(r))\not=f(ar+b)$.
\end{proof}

\begin{theorem}\label{th4}
Let $\phi$ be an injective endomorphism of $K\langle x,y\rangle$ and $p(x,y)\in K\langle x,y\rangle$. If $\phi(p)=x$, then $p$ is a coordinate, namely $(p,q)$ is an automorphism for some polynomial $q(x,y)\in K\langle x.y\rangle$. 
\end{theorem}

\begin{proof}
Let $\phi=(f,g)$. If $\phi$ is an automorphism, then of course $p$ is a coordinate, so we assume $\phi$ is not an automorphism. Since it is injective, $\deg([f,g])\geq 2$. Hence if $p$ is of outer rank 2, by Lemma \ref{lem5}, $\deg(p(f,g))\geq \deg([f,g])\geq 2$. This contradicts to $\deg(p(f,g))=\deg(x)=1$. Then $p(x,y)$ is of outer rank 1. Let $p(x,y)=h(s(x,y))$ where $s$ is a coordinate of $K\langle x,y\rangle$ and $h(u)\in K[u]$. Since $\phi$ is injective, $f,g$ are algebraically independent, and since $s(x,y)\in K\langle x,y\rangle-K$, $s(f,g)\in K\langle x,y\rangle-K$ as well, and hence $\deg(s(f,g))\geq 1$. Then 
$$\deg(p(f,g))=\deg(h(s(f,g)))=\deg_u(h(u))\cdot \deg(s(f,g)).$$
Since $\deg(p(f,g))=\deg(x)=1$, $\deg_u(h(u))=1$, $\deg(s(f,g))=1$, namely $p(x,y)=as+b$ where $a,b\in K$, $a\not=0$. Hence $p$ is a coordinate.
\end{proof}

\noindent We can reprove the following theorem (\cite{C,C1,C2,MLY1}) now.

\begin{theorem}\label{th5}
All the automorphisms of $K\langle x,y\rangle$ are tame.
\end{theorem}

\begin{proof}
Similar to \cite{GY} for zero characteristic case. Here we just use the degree estimate in \cite{LY} instead of \cite{MLY2}.
\end{proof}

\section{Commutators}

\subsection{Restatements of the conjectures}

\noindent For Conjecture \ref{con1}, we can extend it to positive characteristic case naturally. 

\begin{conjecture}\label{con3}
Let $K$ be a field of arbitrary characteristic, $f$ and $g$ be algebraically independent polynomials in $K\langle X\rangle$ such that the homogeneous components of maximal degree of $f$ and $g$ are algebraically dependent. Let $f$ and $g$ generate their own centralizers in $K\langle X\rangle$ respectively. Suppose that $\deg(f)\nmid\deg(g),\deg(g)\nmid\deg(f)$. Then
$$\deg([f,g])>min\{\deg(f),\deg(g)\}.$$
\end{conjecture} 

\noindent However, for Conjecture \ref{con2} it is not true for the positive characteristic case. 

\

\noindent We modify the statement of Conjecture \ref{con2} as follows to extend it to positive characteristic case, and it is easy to verify that it is equivalent to the Conjecture \ref{con2}  for characteristic zero case.
 
\begin{conjecture}\label{con4}
Let K be an arbitrary field and $g\in K\langle X\rangle$ generates its own centralizer. Then to each pair (m,n) where $m,n\in N$, $m,n\geq 2$, $n\nmid m$ and $g^{1/n}$ makes sense in $K((X))$, namely there exists some $h\in K((X))$ such that $g=h^n$, $g^{m/n}\in K((X))$ has a monomial of positive degree containing a negative power of an indeterminate in X.
 \end{conjecture}

\noindent In \cite{DY1}, V. Drensky and J.-T. Yu proved the following theorem

\begin{theorem}\label{th6}
Let $K$ be a field of characteristic zero. As a $K[u_1,u_2]$-bimodule, $K\langle X\rangle$ is a direct sum of three types of submodules:
(1) $K[u]$; (2) $K[u_1,u_2]t$; (3) $K[u_1,u_2]t_1+K[u_1,u_2]t_2$, where

\bigskip

(1) $K[u]$ is generated by 1 where $u_1^p\cdot 1=u_2^p\cdot 1=u^p$;

(2) $u$ is neither a head or a tail of $t$. If $t$ is a tail (a head respectively) of $u$, and $t'$ is the head (the tail respectively) of $u$, then $ut\not=t'u$ ($tu\not=ut'$ respectively).

(3) $t_1$ and $t_2$ are of the same degree and are a proper head and tail of $u$ respectively such that $t_1u=ut_2$. The defining relation of this module is $u_2t_1=u_1t_2$. What's more, there exist $v_1,v_2\in F$ with $v_1v_2\not=v_2v_1$ and a positive integer $k$ such that 
$$u=(v_1v_2)^kv_1,t_1=v_1v_2,t_2=v_2v_1.$$ 
\end{theorem}

\noindent With the similar proof, we can also extend it  to the general case.

\begin{theorem}\label{th7}
Let $K$ be an arbitrary field. As a $K[u_1,u_2]$-bimodule, $K\langle X\rangle$ is a direct sum of three types of submodules:

(1) $K[u]$; (2) $K[u_1,u_2]t$; (3) $K[u_1,u_2]t_1+K[u_1,u_2]t_2$, where

\bigskip

(1) $K[u]$ is generated by 1where $u_1^p\cdot 1=u_2^p\cdot 1=u^p$;

(2) $u$ is neither a head or a tail of $t$. If $t$ is a tail (a head respectively) of $u$, and $t'$ is the head (the tail respectively) of $u$, then $ut\not=t'u$ ($tu\not=ut'$ respectively).

(3) $t_1$ and $t_2$ are of the same degree and are a proper head and tail of $u$ respectively such that $t_1u=ut_2$. The defining relation of this module is $u_2t_1=u_1t_2$. What's more, there exist $v_1,v_2\in F$ with $v_1v_2\not v_2v_1$ and a positive integer $k$ such that 
$$u=(v_1v_2)^kv_1,t_1=v_1v_2,t_2=v_2v_1.$$ 

\end{theorem}

\noindent And we also have the same solution for the equation
 $$[u^m,s]+[u^n,r]=0$$
 in positive characteristic case.

\subsection{Counterexamples}

\noindent For some special case, we have the lemma on radicals for positive characteristic case (\cite{B2}) as follows

\begin{lemma}[Bergman]\label{lem9}
Let $Char(K)=p>0$ and $g\in K((X))$ where $v(g)=h^n$. Then if $p\nmid n$, g is an n-th power in $K((X))$. 
\end{lemma}

\begin{lemma}\label{lem10}
Let $K$ be a field and $Char(K)=2$. Then $g=((xy)^kx)^2+xy+yx\in K((X))$ with $k\geq 2$ has a 2-nd root.
\end{lemma}

\begin{proof}
Let $S=\langle xy,yx,((xy)^kx)^2\rangle$ be the subgroup generated by $xy,yx$ and $((xy)^kx)^2$ in $F=\langle x,y\rangle$ which is the free group generated by $x$ and $y$. Then $S\cap \langle (xy)^kx\rangle=\langle ((xy)^kx)^2\rangle$ since on one hand of course $\langle ((xy)^kx)^2\rangle\in S\cap \langle (xy)^kx\rangle$ and on the other hand, to each $s\in S\cap\langle (xy)^kx\rangle$, $deg(s)$ is even since $s\in S$, and hence if $s=((xy)^kx)^m$, $m$ must be even. Namely $S\cap\langle (xy)^kx\rangle=\langle ((xy)^kx)^2\rangle$. 

\noindent Since $S$ is a subgroup of $F$, it is also well-ordered, and hence $Z/2Z((S))$ makes sense. Now consider $g$ as an element of $Z/2Z((S))$, and according to Bergman, there exists some $e\in Z/2Z((S))$ with $v(e)=1$ such that $supp(ege^{-1})\subseteq C_S(((xy)^kx)^2)$. We have proved above that $C_S(((xy)^kx)^2)=\langle ((xy)^kx)^2\rangle$, and hence $ege^{-1}=((xy)^kx)^2+\sum_{i=0}^{+\infty}a_i((xy)^kx)^{-2i}$ where $a_i=0$ or 1. 

\noindent Let $t=(xy)^kx+\sum_{i=0}^{+\infty}a_i((xy)^kx)^{-i}$, and it is easy to see that $t^2=ege^{-1}$. Let $h=e^{-1}te$, and we get $h^2=g$, namely $g$ is a 2-nd power in $Z/2Z((S))$ and of course in $K((X))$.
\end{proof}

\begin{theorem}\label{th8}
Let $X=\{x,y\},k\geq 2$, and let
$$u=(xy)^kx,v=xy,w=yx,$$
$$f=u^3+r,r=uv+uw+wu,$$
$$g=u^2+s,s=v+w.$$

\noindent Then $(f,g)$ is a counterexample for Conjecture \ref{con3}, and the fraction
$$\frac{\deg([f,g])}{\deg(g)}=\frac{(2k+5)}{(4k+2)}$$ 
is strickly larger than 1/2 and can be made as close to 1/2 as possible by increasing $k$.
\end{theorem}

\begin{proof}
The same as \cite{DY1}.
\end{proof}

\begin{theorem}\label{th9}
Let $X=\{x,y\},k\geq 2$, and let
$$u=(xy)^kx,v=xy,w=yx,$$
$$g=u^2+s,s=v+u.$$
Then $g^{3/2}$ is a counterexample of Conjecture \ref{con4}.
\end{theorem}

\begin{proof}
According to Lemma \ref{lem9} and Lemma \ref{lem10}, $g^{1/2}$ always exists.
The rest are the same as \cite{DY1}.
\end{proof}

\

\section{\bf Acknowledgements}

\noindent Jie-Tai Yu would like to thank Shanghai University and Osaka University
for warm hospitality and stimulating atmosphere during his visit, when part
of the work was done. The authors thank Alexei Belov, Vesselin Drensky and
Leonid Makar-Limanov for their helpful comments and suggestions.

\

\end{document}